\documentclass[doublespacing]{elsart}

\usepackage[applemac]{inputenc} %
\usepackage{amsmath} %
\usepackage{amssymb} %
\usepackage{graphicx} %
\usepackage{stmaryrd} %

\newcommand{\R}{\mathbb{R}} %
\newcommand{\Z}{\mathbb{Z}} %
\newcommand{\N}{\mathbb{N}} %
\newcommand{\Nz}{\N_{0}} %
\newcommand{\C}{\mathbb{C}} %
\newcommand{\suchthat}{: } %
\newcommand{\accol}[1]{{\left\{ #1 \right\}}} %
\newcommand{\set}[2]{\accol{#1 \suchthat #2}} %
\newcommand{\abs}[1]{{\left\lvert #1 \right\rvert}} %
\newcommand{\paren}[1]{{\left( #1 \right)}} %
\newcommand{\floor}[1]{\left\lfloor#1\right\rfloor} 
\newcommand{\ceil}[1]{\left\lceil#1\right\rceil} 
\newcommand{\norm}[2]{{\left\|#2\right\|}_{#1}} %
\newcommand{\scal}[2]{\left\langle#1,  #2\right\rangle} %
\newcommand{\scala}[2]{\left\langle#1,  #2\right\rangle} %
\newcommand{\scalb}[2]{\left\langle\!\left\langle#1,  #2\right\rangle\!\right\rangle} %
\newcommand{\deq}{\mathrel{\mathop:} = } 
\newcommand{\card}{\#} 
\newcommand{\rightinfder}{\underline\partial^+} %
\newcommand{\eps}{\varepsilon} %

\newtheorem{theo}{Theorem} %
\newtheorem{propo}{Proposition}[section] %
\newtheorem{coro}[propo]{Corollary} %
\newtheorem{lemm}[propo]{Lemma} %
\newtheorem{defi}{Definition} %
\newenvironment{proof}{\textit{Proof.} }{\hfill $\square$}

\newcommand{\Snu}{{S}^\nu}
\DeclareMathOperator{\ind}{ind}
\newcommand{\alphamax}{\alpha_{\max}}
\newcommand{\alphamin}{\alpha_{\min}}
\newcommand{\pzero}{p_{0}}

\begin{document}

\begin{frontmatter}

    \title{Advanced topology on the multiscale \\ sequence spaces $\Snu$}

    \author{Jean-Marie Aubry} \ead{jmaubry@math.cnrs.fr}
    \address{Université de Paris-Est, Laboratoire d'Analyse et de
      Mathématiques Appliquées, UMR CNRS 8050, 61 av. du Général de
      Gaulle, F-94010 \textsc{Créteil, France}}
    
    \author{Françoise Bastin} \ead{f.bastin@ulg.ac.be}
    \address{Université de Liège, Département de Mathématiques (B37),
      Grande Traverse, 12, B-4000 \textsc{Liège, Belgique}}

    \begin{abstract}
	We pursue the study of the multiscale spaces $\Snu$ introduced
	by Jaffard in the context of multifractal analysis.  We give
	the necessary and sufficient condition for $\Snu$ to be
	locally $p$-convex, and exhibit a sequence of $p$-norms that
	defines its natural topology.  The strong topological dual of
	$\Snu$ is identified to another sequence space depending on
	$\nu$, endowed with an inductive limit topology.  As a
	particular case, we describe the dual of a countable
	intersection of Besov spaces.
    \end{abstract}


\end{frontmatter}

\section{Introduction}

The advent of multiscale analysis, for instance wavelet techniques,
has shown the need for a new, hierarchical way of organizing
information.  Instead of a sequential data set $(u_{n})_{n\in\N}$, one
has often to work with a tree-indexed data set $(x_{t})_{t \in T}$,
where $T \deq \bigcup\limits_{j \in \Nz} \paren{\Z/ \delta \Z}^{j}$,
$\delta \geq 2$ (without loss of generality, in all this paper we
assume that $\delta = 2$).

In the long tradition of studying sequence
spaces~\cite{Bierstedt:2003yq,Kothe:1969fr}, no special interest was
given to a hierarchical structure of the index set.  But practical
applications, such as multifractal analysis, rely heavily on this
structure: coefficients at different scales~$j$ do not have the same
importance, but within a same scale, they are often interchangeable.
Sequence spaces emphasizing this specific feature have been
introduced: for instance the classical Besov spaces, translated via
the wavelet transform into Besov sequence
spaces~\cite{Triebel:2004qy}, and the intersections of Besov
spaces~\cite{jaffard1}.  However these spaces and their topologies
(Besov-normed or projective limits thereof) provide only an indirect
control, via weighted $l^{p}$ sums, on the asymptotic repartition of
the sizes of the coefficients.  Direct control, such as asking that
the number of coefficients having a certain size be bounded above, was
the motivation which resulted in the more general class of topological
vector spaces called $\Snu$.

\paragraph*{Notations.}

In this paper we consider topological vector spaces (tvs) on the field
$\C$ of complex numbers.  The set of strictly positive natural numbers
is $\N$, and $\Nz \deq \{0\} \cup \N$.  We write $a \wedge b$ and $a
\vee b$ respectively for the minimum and the maximum of $a$ and $b$.
The integer parts are $\floor x \deq \max \set{ z \in \Z }{ z \leq x
}$ and $\ceil x \deq \min \set{ z \in \Z }{ z \geq x}$.  The tree $T$
is canonically identified to the set of indices $\Lambda \deq
\bigcup_{j \in \Nz} \{ j \} \times \{ 0, \dots, 2^j-1 \}$; finally we
define $\Omega \deq \C^\Lambda$ furnished with the pointwise
convergence topology.

\subsection{$\Snu$ spaces}

Inspired by the wavelet analysis of multifractal functions,
Jaffard~\cite{JaffardNotes} introduced spaces $\Snu$ of functions
defined by conditions on their wavelet coefficients.  It was shown
that these spaces are \emph{robust}, that is, they do not depend on
the choice of the wavelet basis (provided the mother wavelet is
sufficiently regular, localized, and has enough zero moments).  So we
can view them, and study their topology, as sequence spaces of wavelet
coefficients; the set of indices $\Lambda$ corresponds to taking the
dyadic wavelet coefficients of a 1-periodic function.

In this context, $\nu$ is a non-decreasing right-continuous function
of a real variable $\alpha$, with values in $\{-\infty\} \cup [0, 1]$,
that is not identically equal to $-\infty$ (an \emph{admissible
profile}).  We define
\begin{equation}
    \alphamin \deq \inf\set{\alpha}{\nu(\alpha) \geq 0} \label{eq:a0}
\end{equation}
and
\begin{equation}
    \alphamax \deq \inf\set{\alpha}{\nu(\alpha) = 1}
    \label{eq:amax}
\end{equation}
with the habitual convention that $\inf \emptyset = +\infty$.  It will
be convenient to remember that this definition implies
\begin{equation}
    \begin{array}{ll}
	\nu(\alpha) = -\infty & \text{ if } \alpha < \alphamin \\
	\nu(\alpha) \in [0, 1) & \text{ if } \alphamin \leq \alpha <
	\alphamax \\
	\nu(\alpha) = 1 & \text{ if } \alpha \geq \alphamax \text{
	(possibly this never happens)}.
    \end{array}
    \nonumber
\end{equation}

\begin{defi}
    The \emph{asymptotic profile} of a sequence $x \in \Omega$ is
    \begin{equation}
	\nu_{x}(\alpha) \deq \lim_{\eps \to 0^+} \
	\limsup_{j\to\infty} \frac{\log\paren{\card\set{k}{\abs{x_{j,
	k}} \geq 2^{-(\alpha+\eps)j}}}}{\log(2^{j})} 
    \end{equation}
\end{defi}
It is easily seen that $\nu_{x}$ is always admissible.

\begin{defi}
    Given an admissible profile $\nu$, we consider the vector space
    \begin{equation}
	\Snu \deq \set{\rule{0ex}{3ex} x \in \Omega}{\nu_{x}(\alpha)
	\leq \nu(\alpha) \quad \forall \alpha \in \R}.
	\label{eq:defSnu}
    \end{equation}
\end{defi}
Heuristically, a sequence $x$ belongs to $\Snu$ if at each scale $j$,
the number of $k$ such that $\abs{x_{j, k}} \geq 2^{-\alpha j}$ is of
order $\leq 2^{\nu(\alpha) j}$.

Finally let us define the \emph{concave conjugate} of $\nu$ is for $p
> 0$
\begin{equation}
    \eta(p) \deq \inf_{\alpha \geq \alphamin} \paren{ \alpha p -
    \nu(\alpha) + 1}.
    \label{eq:concon}
\end{equation}
This function appears in the context of multifractal analysis in the
so-called thermodynamic formalism~\cite{Arneodo:1999fk,jaffard1}.

\subsection{Basic topology}
\label{sec:btp}

Here we summarize the first properties of $\Snu$ established
in~\cite{ABDJ}: There exists a unique metrizable topology $\tau$ that
is stronger than the pointwise convergence and that makes $\Snu$ a
complete tvs.  This topology is separable.  If we define, for $\alpha
\in \R$ and $\beta \in \{-\infty\} \cup [0, +\infty)$, the distance
\begin{equation}
    d_{\alpha, \beta}(x, y) \deq d_{\alpha, \beta}(x-y) \deq \inf
    \set{C \geq 0}{\forall j, \card \accol{\abs{x_{j, k} - y_{j, k}}
    \geq C 2^{-\alpha j}} \leq C 2^{\beta j}} \label{eq:dab}
\end{equation}
(agreeing that $2^{j\beta} \deq 0$ when $\beta = -\infty$) and the
ancillary metric space
\begin{equation}
    E(\alpha, \beta) \deq \set{x \in \Omega}{d_{\alpha, \beta}(x) <
    \infty } 
\end{equation}
then for any sequence $\alpha_{n}$ dense in $\R$ and any sequence
$\eps_{m} \searrow 0$ we have
\begin{equation}
    \Snu = \bigcap_{n, m} E(\alpha_{n}, \nu(\alpha_{n}) +\eps_{m})
    \label{eq:alt}
\end{equation}
and $\tau$ coincides with the projective limit topology (the coarsest
topology which makes each inclusion $\Snu \subset E(\alpha_{n},
\nu(\alpha_{n}) + \eps_{m})$ continuous).  Remark that $E(\alpha,
\beta)$ is never a tvs when $0 \leq \beta < 1$, because the scalar
multiplication is not continuous in $E(\alpha, \beta)$; however it is
continuous in $\Snu$.

When $\beta = -\infty$, the space $E(\alpha, \beta)$ consists in the
set of sequences $x$ such that there exists $C$, for all $(j,k) \in
\Lambda$, $\abs{x_{j,k}} < C 2^{-\alpha j}$.  This space corresponds
to the space of wavelet coefficients of functions in the
Hölder-Zygmund class $C^{\alpha}$, as it was shown by
Meyer~\cite{meyer1}.

Let us also recall the connexion with Besov spaces
(Definition~\ref{defi:Besov} in \S~\ref{sec:Besov} below).  If $\eta$
is the concave conjugate of $\nu$, then
\begin{equation}
    \Snu \subset \bigcap_{\eps > 0} \bigcap_{p > 0} b^{\eta(p)/p -
    \eps}_{p, \infty} \label{eq:SnuB}
\end{equation}
with equality if and only if $\nu$ is concave.

\subsection{Prevalent properties}

Some measure-related properties of $\Snu$: Given $(\psi_{j,k})$ an
$L^{\infty}$-normalized orthogonal wavelet basis of $L^2(\R/\Z)$ and
assuming that $\alphamin > 0$, let $f_{x} \deq \sum_{j,k} x_{j,k}
\psi_{j,k}$ and let $d_{f_{x}} : \alpha \mapsto
\dim_{H}(\set{t}{h_{f_{x}}(t) = \alpha})$ be its Hausdorff spectrum of
Hölder singularities.  Then for $x$ in a
prevalent~\cite{hunt92:_preval} subset of $\Snu$, for all $\alpha \in
\R$, we have $\nu_{x}(\alpha) = \nu(\alpha)$ and $d_{f_{x}}(\alpha) =
\alpha \sup_{ \alpha'\in (0, \alpha]} \frac{\nu( \alpha')}{ \alpha'}$
if $ \alpha \leq h_{\max}$, $d_{f_{x}}(\alpha) = -\infty$ else (here
$h_{\max} \deq \inf_{\alpha > \alphamin} \frac{\alpha}{\nu(\alpha)}$).
We refer to~\cite{Aubry:zg} for the details.

\subsection{Outline of the results}

Our goal in this paper is to establish advanced topological properties
such as local convexity and duality for the spaces $\Snu$.  We shall
see that these properties, unlike those we recalled in
\S~\ref{sec:btp}, depend in a subtle way on the profile~$\nu$.

In \S~\ref{sec:prel} we give some definitions and preliminary results.
In \S~\ref{sec:locg} we study local convexity.  Having shown that
$\Snu$ is never $p$-normable (\S~\ref{sec:nono}), we give the
necessary and sufficient condition for local $p$-convexity in
\S~\ref{sec:locconv}.  A set of $p$-norms inducing the topology is
presented in \S~\ref{sec:semin}.  The last section
(\S~\ref{sec:duality}) of this article addresses the identification of
the topological dual of $\Snu$.  This dual $(\Snu)'$ turns out
(\S~\ref{sec:topodual}) to be a union of sequence spaces just smaller
than another space $S^{\nu'}$, where $\nu'$ can be derived from $\nu$
in a way shown in \S~\ref{sec:dualpro}.  In \S~\ref{sec:topodl} we
prove that the strong topology on $(\Snu)'$ is the same as the
inductive limit topology on this union and deduce the condition for
reflexivity.  The particular case where $\Snu$ is an intersection of
Besov spaces is detailed in \S~\ref{sec:dualB}.

\section{Preliminaries}
\label{sec:prel}

\subsection{Right-inf derivative}

We shall see that the local convexity of $\Snu$ depends on $\nu$, and
more precisely on its right-inf derivative defined, whenever
$\nu(\alpha) > -\infty$, as
\begin{equation}
    \rightinfder \nu(\alpha) \deq \liminf\limits_{h \to 0^+}
    \frac{\nu(\alpha+h)-\nu(\alpha)}{h} \label{eq:rid}
\end{equation}
for which holds an equivalent of the mean value inequality.

\begin{lemm}
    \label{lemm:semitaf}
    Let $\nu$ be an admissible profile and $p > 0$.  The following
    assertions are equivalent.
    \begin{enumerate}
	\item \label{it:un} For all $\alphamin \leq \alpha <
	\alphamax$, $\rightinfder \nu(\alpha) \geq p$;
	
	\item \label{it:do} For all $\alphamin \leq \alpha' \leq
	\alpha < \alphamax$, $\nu(\alpha) - \nu(\alpha') \geq
	p(\alpha-\alpha')$.
    \end{enumerate}
\end{lemm}

\begin{proof}
    The~\ref{it:do} $\Rightarrow$~\ref{it:un} part is obvious.
    Conversely, let $f_{n}(x) \deq n(\nu(x+\frac{1}{n})-\nu(x))$.
    Note that $\liminf_{n} f_{n}(x) \geq \rightinfder \nu(x) \geq
    p$ when $\alpha' \leq x \leq \alpha$.  So
    \begin{align*}
	p(\alpha-\alpha') &\leq \int_{\alpha'}^{\alpha}
	\liminf_{n} f_{n}(x) d x \\
	\intertext{By Fatou's lemma, } &\leq \liminf_{n}
	\int_{\alpha'}^{\alpha} f_{n}(x) d x \\
	&\leq \liminf_{n} \ n
	\paren{\int_{\alpha}^{\alpha+\frac{1}{n}} \nu(x) d x -
	\int_{\alpha'}^{\alpha'+\frac{1}{n}} \nu(x) d x } \\
	&\leq \nu(\alpha) -\nu(\alpha')
    \end{align*}
    because $\nu$ is right-continuous.
\end{proof}

\subsection{$p$-norm and local $p$-convexity}
\label{sec:defp}

Let us recall a couple of definitions from
Jarchow~\cite{Jarchow:1981sp}, which generalize the notions of norm
and local convexity (these corresponding to the case $p = 1$).

\begin{defi}
    Let $X$ be a vector space and $0 < p \leq 1$.  A map $q : X \to
    [0, +\infty)$ is a \emph{$p$-seminorm} if $q(\lambda x) =
    \abs{\lambda} q(x)$ for all $\lambda \in \C$, $x \in X$ and if
    \begin{math}
	q(x+y)^{p} \leq q(x)^{p} + q(y)^{p} \nonumber
    \end{math}
    for all $x, y \in X$.
    
    If in addition $q(x) = 0$ only if $x = 0$ then $q$ is called a
    \emph{$p$-norm}.
\end{defi}

\begin{defi}
    Let $0 < p \leq 1$.  A subset $K$ of a vector space $X$ is
    \emph{$p$-convex} if for all $x_{1}, \dots, x_{N} \in K$ and
    $\theta_{1}, \dots, \theta_{N}$ such that $\sum_{n} \theta_{n}^{p}
    = 1$, the \emph{$p$-convex combination} $\sum_{n} \theta_{n}
    x_{n}$ belongs to $K$.
    
    We say that $K$ is \emph{absolutely $p$-convex} if in addition it
    is \emph{circled} (or \emph{balanced}), that is, $\lambda x \in K$
    whenever $x \in K$ and $\abs{\lambda} \leq 1$.
    
    The tvs $X$ is \emph{locally $p$-convex} if it has a basis of
    absolutely $p$-convex neighbourhoods of 0.
\end{defi}

Clearly, a $p$-normed space is locally $p$-convex.  For instance the
sequence space $l^p$, $p > 0$, is $1 \wedge p$-normed thus $1 \wedge
p$-locally convex.

When $X = \bigcap_{n} X_{n}$ endowed with the projective limit
topology, $X$ is locally $p$ convex if and only if for each $n$, for
each $0$-neighbourhood $V$ in $X_{n}$, there exist $n_{1}, \dots,
n_{L}$ and $0$-neighbourhoods $U_{1}, \dots, U_{L}$ in $X_{n_{1}},
\dots, X_{n_{L}}$ such that any $p$-convex combination of elements of
$U \deq \bigcap_{l} U_{l}$ stays in $V$.

\subsection{Besov spaces}
\label{sec:Besov}

To end this section, and to prepare the results of the next one, we
elucidate the link between (sequence) Besov spaces and the ancillary
spaces $E(\alpha, \beta)$.

\begin{defi}
    \label{defi:Besov}
    For $\alpha \in \R$, $0 < p < \infty$, the $b^{\alpha}_{p,
    \infty}$ Besov $1\wedge p$-norm of a sequence $x$ is given by
    \begin{equation}
	\norm{b^{\alpha}_{p, \infty}}{x} \deq \sup_{j \in \Nz}
	2^{(\alpha - \frac{1}{p}) j} \paren{\sum_{k = 0}^{2^j-1}
	\abs{x_{j, k}}^p}^{\frac{1}{p}} \label{eq:pnb}
    \end{equation}
    and if $p = \infty$,
    \begin{equation}
	\norm{b^{\alpha}_{\infty, \infty} }{x} \deq \sup_{j \in \Nz}
	\sup_{0 \leq k < 2^{j}} 2^{\alpha j} \abs{x_{j, k}}.
    \end{equation}
\end{defi}

It is easy to check that, being modelled on $l^p$, the space
$b^{\alpha}_{p, \infty}$ is $1\wedge p$-normed, thus $1\wedge
p$-convex (and not better).  Note that $b^{\alpha}_{\infty, \infty} =
E(\alpha, -\infty)$.

\begin{lemm}
    \label{lemm:contBesov}
    Let $0 < p < \infty$ and $s \in \R$.  If $\beta \geq \alpha p + 1
    - s$, then for all $x \in b^{s/p}_{p, \infty}$,
    \begin{equation}
	d_{\alpha, \beta}(x) \leq \norm{b^{s/p}_{p,
	\infty}}{x}^{\frac{p}{p+1}}.
	\label{eq:contBesov}
    \end{equation}
\end{lemm}

\begin{proof}
    Let $C \deq \norm{b^{s/p}_{p, \infty}}{x}$.  If there exists a $j$
    such that
    \begin{equation}
	\card\set{k}{\abs{x_{j, k}} \geq C^{\frac{p}{p+1}} 2^{-\alpha
	j} } > C^{\frac{p}{p+1}} 2^{\beta j} \nonumber
    \end{equation}
    then $C \geq 2^{\frac{s-1}{p}j} \paren{\sum_{k} \abs{x_{j,
    k}}^p}^{\frac1p} > C 2^{\frac{\beta+s-1-\alpha p}{p} j} \geq C$, a flagrant
    contradiction.  So it must be that for all $j$,
    $\card\set{k}{\abs{x_{j, k}} \geq C^{\frac{p}{p+1}} 2^{-\alpha j}
    } \leq C^{\frac{p}{p+1}} 2^{\beta j}$.  This shows that
    $d_{\alpha, \beta}(x) \leq C^{\frac{p}{p+1}}$.
\end{proof}

It follows that $b^{s/p}_{p, \infty} \subset E(\alpha, \beta)$
continuously, but the converse inclusion is never true.

Another Besov embedding will be useful.
\begin{lemm}
    \label{lemm:bz}
    If $0 < p \leq p'$ and $\alpha \in \R$, then for all $x \in
    b^{\alpha}_{p', \infty}$,
    \begin{equation}
	\norm{b^{\alpha}_{p, \infty}}{x} \leq \norm{b^{\alpha}_{p',
	\infty}}{x}.
    \end{equation}
\end{lemm}

\begin{proof}
    By Hölder's inequality, if $p \leq p'$,
    \begin{equation}
	\norm{b^{\alpha}_{p, \infty}}{x} = 2^{(\alpha - \frac{1}{p})
	j} \paren{\sum_{k = 0}^{2^{j}-1} \abs{x_{j,
	k}}^{p}}^{\frac{1}{p}} \leq 2^{(\alpha - \frac{1}{p}) j}
	2^{(\frac{1}{p}-\frac{1}{p'}) j} \paren{ \sum_{k =
	0}^{2^{j}-1} \abs{x_{j, k}}^{p'}}^{\frac{1}{p'}} =
	\norm{b^{\alpha}_{p', \infty}}{x}.  \nonumber
    \end{equation}
\end{proof}

\paragraph*{Remark.}

This embedding uses specifically the fact that $0 \leq k 2^{-j} < 1$
(Besov space on a compact domain).  It can be compared to Lemma 8.2
of~\cite{ABDJ}, which is valid on any domain (e.g. $k \in \Z$): When
$p' \leq p$ and $\alpha - \frac{1}{p} \leq \alpha' - \frac{1}{p'}$ we
also have $\norm{b^{\alpha}_{p, \infty}}{x} \leq
\norm{b^{\alpha'}_{p', \infty}}{x}$.

\section{Local geometry of $\Snu$}
\label{sec:locg}

Here we apply the definitions of~\S~\ref{sec:defp} to $\Snu$ spaces.
We shall see that $\Snu$ is never $p$-normable, but that it is locally
$p$-convex for $p$ depending on $\nu$.

\subsection{Non normability}
\label{sec:nono}

\begin{propo}
    The tvs $\Snu$ is not $p$-normable for any $p > 0$.
\end{propo}

\begin{proof}
    Suppose that $q$ is a $p$-norm defining the topology of $\Snu$.
    Then there are $\alpha_{l} \in \R$, $\eps_{l} > 0$ ($l = 1, \dots,
    L$) and $\delta_{0} > 0$ such that
    \begin{equation}
	\bigcap_{l = 1}^{L} U_{l} \subset B \nonumber
    \end{equation}
    where $B \deq \set{x \in \Snu}{q(x) \leq 1}$ and $U_{l} \deq
    \set{x \in \Snu}{d_{\alpha_{l}, \nu(\alpha_{l}) + \eps_{l}}(x)
    \leq \delta_{0}}$.  We assume $\alpha_{1} < \dots < \alpha_{L}$
    and $\delta_{0} < 1$.
    
    \underline{First case:} There is $l$ such that $\alpha_{l} <
    \alphamin$.  Let $n$ be the largest integer satisfying this.  For
    all $l \leq n$, $\nu(\alpha_{l}) = -\infty$ and thus
    $d_{\alpha_{l}, \nu(\alpha_{l}) + \eps_{l}}(x) = \sup_{(j, k) \in
    \Lambda} 2^{\alpha_{l} j} \abs{x_{j, k}}$.  For $m \in \Nz$ we
    define the sequence $x^{m} \in S^{\nu}$ by setting at scale $m$
    exactly one coefficient equal to $2^{-\alpha_{n} m} \delta_{0}$.
    We claim that $x^{m} \in B$ for sufficiently large $m$.  Indeed,
    if $l \leq n$ then $d_{\alpha_{l}, \nu(\alpha_{l}) +
    \eps_{l}}(x^{m}) \leq \delta_{0}$ for all $m$; if $l > n$ we have
    $\nu(\alpha_{l}) + \eps_{l} > 0$ hence $d_{\alpha_{l},
    \nu(\alpha_{l}) + \eps_{l}}(x^{m}) \leq \delta_{0}$ as soon as $m
    \geq -\frac{\log_{2}(\delta_{0})}{\nu(\alpha_{l}) + \eps_{l}}$.
    So $x^{m} \in \bigcap_{l = 1}^{L} U_{l} \subset B$.
    
    Let us now consider $\alpha_{n} < \alpha' < \alphamin$.  If $B'$
    denotes the unit ball in $E(\alpha', -\infty)$, our hypothesis
    that the topology of $\Snu$ stems from the $p$-norm $q$ implies
    that there exists $\lambda > 0$ such that $\lambda B \subset B'$.
    This would imply that $\lambda x^{m} \in B'$ for all $m$, a
    contradiction because $d_{\alpha', -\infty}(x^{m}) =
    2^{(\alpha'-\alpha_{n}) m} \to \infty$.
    
    \underline{Second case:} $\alpha_{l} \geq \alphamin$ for all $l$.
    We chose $\alpha'' < \alpha' < \alphamin$ and define the sequence
    $x^{m} \in S^{\nu}$ by setting at scale $m$ exactly one
    coefficient equal to $2^{-\alpha'' m} \delta_{0}$; the rest of the
    proof is identical to the sub-case $l > n$ above.
\end{proof}

\subsection{Local convexity}
\label{sec:locconv}

The convexity index will be
\begin{equation}
    \pzero \deq 1 \wedge \inf_{{0 \leq \nu(\alpha) < 1}} \rightinfder
    \nu (\alpha) \label{eq:pzero}
\end{equation}
(we recall that $0 \leq \nu(\alpha) < 1$ is equivalent to $\alphamin
\leq \alpha < \alphamax$).

How does this number come into play?  When $\nu$ is concave, we know
from~\eqref{eq:SnuB} that $\Snu$ is an intersection of (sequence)
Besov spaces $b^{\eta(p)/p - \eps}_{p, \infty}$ for $\eps > 0$ and $p
> 0$, $\eta$ being the convex conjugate~\eqref{eq:concon} of $\nu$.
Suppose that $0 < p \leq \pzero$, and observe that by
Lemma~\ref{lemm:semitaf}, $\frac{\eta(p)}{p} = \frac{\eta(\pzero
)}{\pzero} ( = \alphamax) < \infty$.  Then Lemma~\ref{lemm:bz} leads
to $b^{\eta(p)/p - \eps}_{p, \infty} \supset b^{\eta(\pzero)/\pzero -
\eps}_{\pzero, \infty}$.  So in fact
\begin{equation}
    \Snu = \bigcap_{\eps > 0} \bigcap_{p \geq \pzero} b^{\eta(p)/p -
    \eps}_{p, \infty} \nonumber
\end{equation}
an intersection of spaces at least $\pzero$-convex.  This idea leads
to the general case.

\begin{theo}
    \label{theo:locconv}
    Let $\pzero$ be defined by~\eqref{eq:pzero}
    \begin{itemize}
	\item If $\pzero > 0$, then $\Snu$ is locally $\pzero
	$-convex.  \item If $\pzero < 1$, then $\Snu$ is not locally
	$p$-convex for any $p > \pzero$.
    \end{itemize}
\end{theo}

\begin{proof}
    We first prove the second point.  Suppose that $\pzero < p < 1$.
    Our purpose is to find a neighbourhood $V$ in $\Snu$, for instance
    the unit ball in some $E(\alpha', \nu(\alpha') + \eps)$, such that
    it cannot contain the $p$-convex hull of any $0$-neighbourhood
    $U$.
    
    By definition of $\pzero$, there exist $\eps > 0$ and $\alphamin
    \leq \alpha < \alpha'$ such that $\nu(\alpha') + \eps < 1$ and
    $\nu(\alpha') - \nu(\alpha) + \eps < p (\alpha' - \alpha)$.  For
    short we shall write
    \begin{equation}
	s \deq \alpha'-\alpha \nonumber
    \end{equation}
    and
    \begin{equation}
	t \deq \nu(\alpha') - \nu(\alpha) + \eps.  \nonumber
    \end{equation}
    Thanks to the right-continuity of $\nu$, $\eps$ and $s$ can be
    taken small enough so that
    \begin{equation}
	\frac{p}{p+1} \paren{s + t} < 1 - \nu(\alpha).
	\label{eq:sml}
    \end{equation}
    
    Assume that $U$ is a $p$-convex $0$-neighbourhood in $\Snu$ such
    that $$U\subset V=\left\{\rule{0ex}{3ex} x\in \Snu:
    d_{\alpha',\nu(\alpha')+\varepsilon}(x)<1\right\} $$
    and let $z\in\Snu$ be such that $z_{j, k}=2^{-\alpha j}$ for
    $\lfloor 2^{\nu(\alpha)j}\rfloor$ values of $k$ at each scale $j$.
    Since $\Snu$ is a tvs, there is $\lambda>0$ such that $\lambda
    z\in U$; moreover, the special structure of the topology of $\Snu$
    allows to say that this remains true if some coefficients have
    been set to $0$ or moved within the scale.

    Now, let $N \in \mathbb{N}$ be fixed.  If $j_0$ is the smallest
    integer such that $2^{j_0} \ge N 2^{\nu(\alpha) j_0}$, we
    construct $x_1,\ldots, x_N$ having, at each scale $j \ge j_0$
    disjoint sets of cardinal $\lfloor 2^{\nu(\alpha)j} \rfloor$ of
    coefficients equal to $\lambda 2^{-\alpha j}$ (and the others are
    $0$).  These sequences all belong to $U$.  We then form the
    $p$-convex combination 
    \begin{equation*}
	x \deq N^{-\frac{1}{p}} \sum_{n=1}^N x_n.
    \end{equation*}
    Note that, at each scale $j \ge j_0$, the sequence $x$ has
    $N\lfloor 2^{\nu(\alpha)j} \rfloor = C'(N,j) 2^{\nu(\alpha') +
    \varepsilon)j}$ coefficients of size equal to
    $N^{-\frac{1}{p}}\lambda 2^{-\alpha j}=C(N,j)2^{-\alpha'j}$, with
    $C(N,j) \deq \lambda N^{-\frac{1}{p}}2^{sj}$ and $C'(N,j) \deq
    N2^{-tj}\frac{\lfloor 2^{\nu(\alpha)j}\rfloor}{2^{\nu(\alpha)j}}$.

    Let us focus on the scale
    \begin{math}
	j \deq \ceil{\frac{\frac{p+1}{p} \log_{2}\paren{N} - \log_{2}
	\paren{\lambda}}{s + t}}.  \nonumber
    \end{math}
    Because of~\eqref{eq:sml}, we check that this $j \geq j_{0}$.
    Furthermore since $j \geq \frac{\frac{p+1}{p} \log_{2}\paren{N} -
    \log_{2} \paren{\lambda}}{s + t}$,
    \begin{align*}
	C(N, j) & \geq \lambda N^{-\frac{1}{p}}
	2^{\frac{s}{s+t}\paren{\frac{p+1}{p} \log_{2}(N) -\log_{2}(
	\lambda )} }\\
	& \geq \lambda^{\frac{t}{s+t}} N^{\frac{p s - t}{p (s + t)}}
    \end{align*}
    and since $j < 1 + \frac{\frac{p+1}{p} \log_{2}\paren{N} -
    \log_{2} \paren{\lambda}}{s + t}$,
    \begin{align*}
	C'(N, j) & > \frac{N}2 2^{-t -\frac{t}{s+t}
	\paren{\frac{p+1}{p} \log_{2}(N) -\log_{2}( \lambda )} } \\
	& > 2^{-t-1} \lambda^{\frac{t}{s+t}} N^{\frac{p s - t}{p (s +
	t)}}.
    \end{align*}
    By the hypothesis on $p$, the exponent of $N$ is strictly positive
    and this shows that $d_{\alpha', \nu(\alpha') + \eps}(x)$ can be
    arbitrarily large with $N$.  In particular, no matter how small
    the neighbourhood $U$ is, there exists a $p$-convex combination of
    elements of $U$ which does not belong to $V$, meaning the unit
    ball in $E(\alpha', \nu(\alpha') + \eps)$.  So $\Snu$ is not
    locally $p$-convex.
    
    The proof of the first point boils down to this: Let $M > 0$,
    $\eps > 0$ and $\alpha$ be fixed, we want to find a
    $0$-neighbourhood $U$ in $\Snu$ such that any $\pzero$-convex
    combination of elements $x_{1}, \dots, x_{N} \in U$ will stay in
    $V \deq \set{x \in \Snu}{d_{\alpha, \nu(\alpha) + \eps}(x) \leq
    M}$.
    
    The two cases $\nu(\alpha) = -\infty$ and $\nu(\alpha) = 1$ are
    both trivial because $E(\alpha, -\infty) = b^{\alpha}_{\infty,
    \infty}$ and $E(\alpha, 1 + \eps) = \Omega$ (the set of all
    sequences with the topology of pointwise convergence,
    see~\cite{Aubry:2007uq}) are locally convex.
    
    From now on we assume that $0 \leq \nu(\alpha) < 1$.  Let $L \deq
    \ceil{(\alpha - \alphamin) \frac{2 \pzero}\eps}$, $\lambda \deq
    \frac{M}{L+2}$ and for $-1 \leq l \leq L$ let $\alpha_{l} \deq
    \alphamin + \frac{\eps}{2 \pzero} l$ and $\nu_{l} \deq
    \nu(\alpha_{l}) + \frac{\eps}{2}$.  Note that $\alpha_{L} \geq
    \alpha$.  We construct
    \begin{equation}
	U \deq \bigcap_{l = -1}^{L} U_{l} \label{eq:U}
    \end{equation}
    where
    \begin{equation}
	U_{l} \deq\set{x \in \Snu}{d_{\alpha_{l}, \nu_{l}}(x) <
	\lambda} \nonumber
    \end{equation}
    (in particular, since $\nu_{-1} = -\infty$, $U_{-1} \subset
    \set{x}{\forall j, k, \abs{x_{j, k}} < \lambda 2^{-\alpha_{-1}j}
    }$).
    
    For an arbitrary $N \in \N$, let $x_{1}, \dots, x_{N} \in U$ and
    $\theta_{1}, \dots, \theta_{N}$ be the coefficients of a
    $\pzero$-convex combination $x \deq \sum_{n = 1}^{N} \theta_{n}
    x_{n}$.  We split each $x_{n}$ as $x_{n} = \sum_{l = 0}^{L+1}
    x_{n}^{l}$, where for $0 \leq l \leq L$, $x_{n}^{l}$ receives the
    coefficients $\lambda 2^{-\alpha_{l}j} < \abs{x_{j, k}} \leq
    \lambda 2^{-\alpha_{l-1}j}$ (the others are set to $0$) and
    $x_{n}^{L+1}$ receives the coefficients $\abs{x_{j, k}} \leq
    \lambda 2^{-\alpha_{L}j}$; since $x_{n} \in U_{-1}$ it has no
    coefficients $\abs{x_{j, k}} > \lambda 2^{-\alpha_{-1}j}$.  Once
    this is done, we do the $\pzero$-convex combinations
    \begin{equation}
	x^{l} \deq \sum_{n = 1}^{N} \theta_{n} x_{n}^{l}.  \nonumber
    \end{equation}
    Remark that $x_{n} \in U$ implies that each $x_{n}^{l} \in U$ and
    a fortiori $x_{n}^{l} \in U_{l}$.  Let us contemplate two cases.
    
    \underline{When $0 \leq l \leq L$:} Since $x_{n}^{l}$ is in
    $U_{l}$ and has only coefficients $\abs{x_{j, k}} > \lambda
    2^{-\alpha_{l}j}$, the cardinal of the set of non-zero
    coefficients at scale $j$ in $x_{n}^{l}$ is smaller than $\lambda
    2^{\nu_{l} j}$, and these coefficients are all bounded from above
    by $\lambda 2^{-\alpha_{l-1}j}$.  It follows that if $s - 1
    +\nu_{l} - \alpha_{l-1} p \leq 0$, then $\norm{b^{s/p}_{p,
    \infty}}{x_{n}^{l}} \leq \lambda^{\frac{p+1}{p}}$.  Actually we
    take $p = \pzero$ and $s = s_{l} \deq 1 + \alpha_{l-1} \pzero -
    \nu_{l}$.  Since $b^{s_{l}/\pzero}_{\pzero, \infty}$ is $\pzero
    $-convex, $\norm{b^{s_{l}/\pzero}_{\pzero, \infty}}{x^{l}} \leq
    \lambda^{\frac{\pzero +1}{\pzero}}$ as well.
    
    Thanks to Lemma~\ref{lemm:semitaf} and the definition of $\pzero$,
    \begin{align*}
	\alpha \pzero + 1 - s_{l} - \nu(\alpha) - \eps & = \pzero
	(\alpha - \alpha_{l-1}) + \nu_{l} - \nu(\alpha) - \eps \\
	& = \pzero (\alpha - \alpha_{l} + \frac{\eps}{2 \pzero}) +
	\nu(\alpha_{l}) + \frac{\eps}{2} - \nu(\alpha) - \eps \\
	& = \pzero (\alpha - \alpha_{l}) + \nu(\alpha_{l}) -
	\nu(\alpha) \\
	&\leq 0.
    \end{align*}
    Thus Lemma~\ref{lemm:contBesov} is applicable with $\beta =
    \nu(\alpha) + \eps$, and we get that
    \begin{equation}
	d_{\alpha, \nu(\alpha) + \eps}(x^{l}) \leq
	\norm{b^{s_{l}/\pzero}_{\pzero,
	\infty}}{x^{l}}^{\frac{\pzero}{\pzero + 1}} \leq \lambda.
	\nonumber
    \end{equation}
    
    \underline{When $l = L+1$}, simply notice that what remains in
    each $x_{n}^{L+1}$ are coefficients $\abs{x_{j, k}} \leq \lambda
    2^{-\alpha j}$.  The $\pzero$-convex combination $x^{L+1}$ will
    have \emph{a fortiori} only coefficients $\abs{x_{j, k}} \leq
    \lambda 2^{-\alpha j}$ and this shows that $d_{\alpha, \nu(\alpha)
    + \eps}(x) \leq \lambda$ (indeed any $C > \lambda$ satisfies the
    condition in~\eqref{eq:dab}, so their infimum is $\leq \lambda$).
    
    Finally,
    \begin{equation}
	d_{\alpha, \nu(\alpha) + \eps}(x) \leq \sum_{l = 0}^{L+1}
	d_{\alpha, \nu(\alpha) + \eps}(x^{l}) \leq (L+2) \lambda = M.
	\nonumber
    \end{equation}
    We have proved that $\Snu$ is locally $\pzero$-convex.
    
\end{proof}

\begin{coro}
    \label{coro:frechet}
    The space $\Snu$ is a Fréchet space if and only if $\rightinfder
    \nu (\alpha) \geq 1$ for all $\alpha \in [\alphamin,
    \alphamax)$.
\end{coro}

\begin{proof}
    We already know that $\Snu$ is a metrizable and complete
    tvs~\cite{ABDJ}.  We apply the previous theorem with $\pzero = 1$.
\end{proof}

\subsection{Another description of the topology}
\label{sec:semin}

The topology of a Fréchet space can always be defined by a sequence of
seminorms.  When the space is only $\pzero$-convex, the seminorms are
naturally to be replaced by $\pzero$-seminorms.  The interesting
feature in the case of $\Snu$ is that they can be made explicit as
$\pzero$-norms interpolating two Besov spaces.

\newcommand{\thebesov}{b^{s}_{\pzero, \infty}}
\begin{theo}
    \label{theo:seminorms}
    If $\pzero > 0$, the topology of $\Snu$ is induced by the family of
    norms $\norm{b^{\alphamin - \eps}_{\infty, \infty}}{x}$ together
    with the $\pzero$-norms
    \begin{equation}
	\norm{\alpha, \eps}{x} \deq \inf \set{\norm{\thebesov}{x'} +
	\norm{ b^{\alpha}_{\infty, \infty}}{x''} }{
	x' + x'' = x }
    \end{equation}
    where $\alpha \in [\alphamin, \alphamax)$, $\eps > 0$ and $s
    \deq \alpha + \frac{1-\nu(\alpha)}{\pzero} - \eps$.
    
    This family of $\pzero$-norms can be made countable by taking
    sequences $(\alpha_{n})$ dense in $[\alphamin, \alphamax)$
    and $(\eps_{m}) \to 0^{+}$.
\end{theo}

To illustrate this theorem, notice (Figure \ref{fig:seminormnu}) that
the hypograph of $\nu$ is the intersection of the sets (parameters
domains) of $(\tilde\alpha, \beta)$ such that a sequence having
$2^{\beta j}$ coefficients $= 2^{-\tilde\alpha j}$ belongs to
$\thebesov + b^{\alpha}_{\infty, \infty}$.

\begin{figure}[ht] 
    \begin{center}
	\includegraphics[width = 0.75\textwidth]{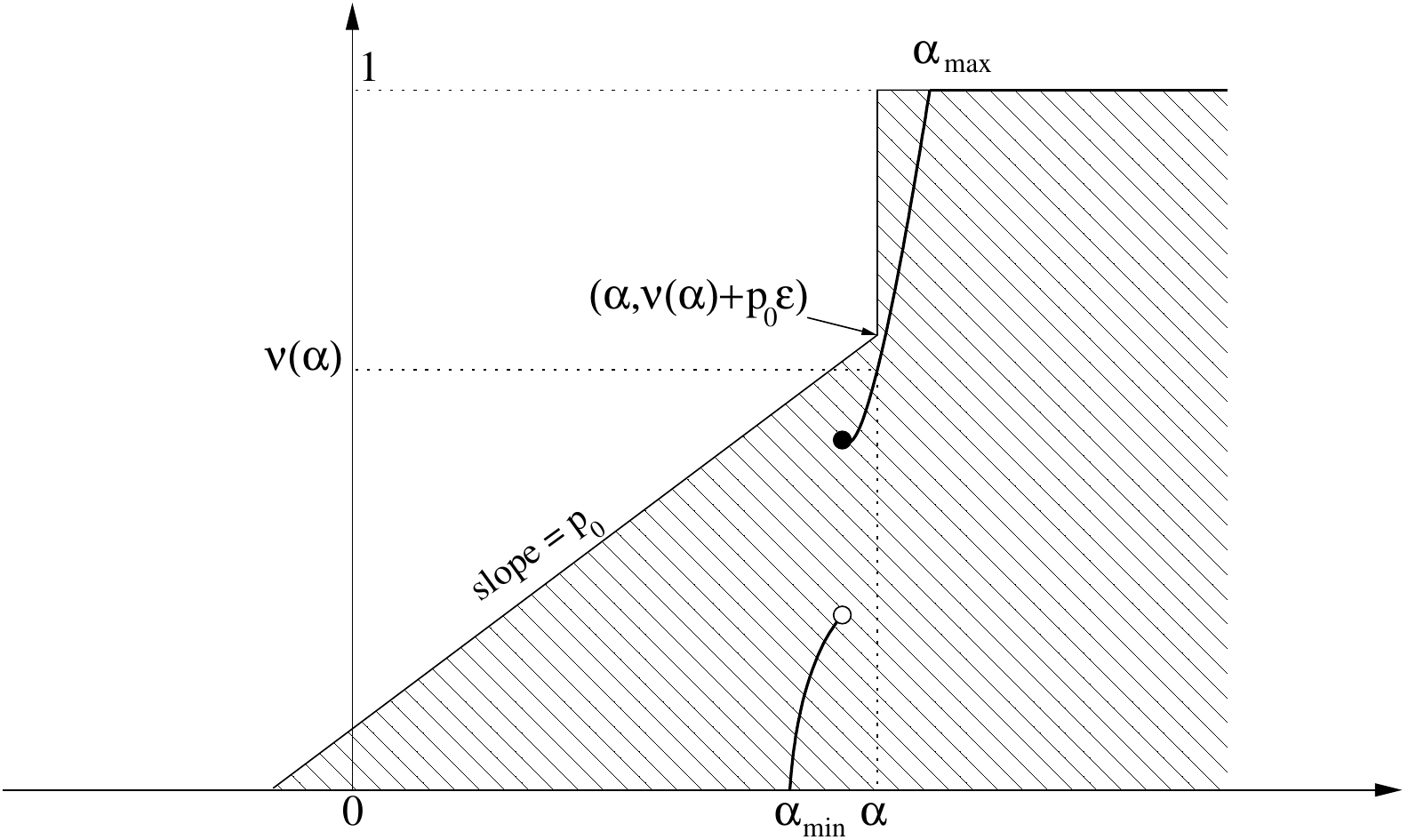}
    \end{center}
    \caption{$\nu(\alpha)$ and the parameters domain of $\thebesov +
    b^{\alpha}_{\infty, \infty}$}
    \label{fig:seminormnu}
\end{figure}

\begin{proof}
    Let $\alpha \in [\alphamin, \alphamax)$ and $\eps > 0$.  As
    in the proof of the first point of Theorem~\ref{theo:locconv}, we
    define $L$, $\alpha_{l}$ and $\nu_{l}$ for $-1\le l \le L-1$ and
    $\alpha_L \deq \alpha,\ \nu_L \deq \nu(\alpha_{L}) +
    \frac{\varepsilon}{2}$.  Let $0 < \delta < 1$ be fixed and let
    $\lambda \deq \frac{\delta}{L+2}$.  The neighbourhoods $U_{l}$
    and, for $x \in \bigcap_{l=-1}^{L} U_{l}$, the splitting $x =
    \sum_{l=0}^{L+1} x^{l}$ are unchanged.
    
    It is clear that, since $\alpha_{L} \geq \alpha$, $x'' \deq
    x^{L+1}$ belongs to $b^{\alpha}_{\infty, \infty}$ and moreover
    $\norm{b^{\alpha}_{\infty, \infty}}{x''} \leq \lambda$.
    On the other hand, with $s \deq \alpha + \frac{1 -
    \nu(\alpha)}{\pzero} - \eps$, for $0 \leq l \leq L$, each $x^{l}$
    belongs to $b^{s}_{\pzero, \infty}$ and $\norm{b^{s}_{\pzero,
    \infty}}{x^l} \leq \lambda^{\frac{\pzero+1}{\pzero}} \leq \lambda$.
    So finally $x' \deq \sum_{l=0}^{L} x^{l}$ satisfies
    $\norm{b^{s}_{\pzero, \infty}}{x'} \leq (L+1) \lambda$.  We have
    proved that by our choice of $\lambda$,
    \begin{equation*}
	\bigcap_{l=-1}^{L} U_{l} \subset \set{x \in
	\Snu}{\norm{\alpha, \eps}{x} \leq \delta}
    \end{equation*}
    in other words, that $x \mapsto \norm{\alpha,\eps}{x} $ is a
    continuous $\pzero$-norm on $S^{\nu}$.

    Now, let us show that given $\alpha \in [\alphamin, \alphamax),
    \eps > 0$ and $0 < \delta < 1$, we have
    \begin{equation*}
	\set{x\in\Snu}{ \norm{\alpha, \eps}{x} <
	\paren{\frac\delta2}^{\frac{\pzero+1}{\pzero}} } \subset
	\set{\rule{0ex}{3ex}x\in\Snu}{d_{\alpha,\nu(\alpha)+\eps}(x)
	\leq \delta }.
    \end{equation*}    
    If $\norm{\alpha, \eps}{x} <
    \paren{\frac\delta2}^{\frac{\pzero+1}{\pzero}}$, then there exist $x',
    x'' \in \Snu$ such that $x = x' + x''$, $\norm{b^{s}_{\pzero,
    \infty}}{x'} < \paren{\frac\delta2}^{\frac{\pzero+1}{\pzero}}$ and
    $\norm{b^{\alpha}_{\infty, \infty}}{x''} <
    \paren{\frac\delta2}^{\frac{\pzero+1}{\pzero}} \leq \frac\delta2$.
    
    Using the inequality between distances   $d_{\alpha,\beta}$ and
    Besov $p$-norms (Lemma~\ref{lemm:contBesov}) and the fact that
    $\sup_{j,k} 2^{\alpha j} |x''_{j, k}| \leq \frac{\delta}{2}$ implies
    $d_{\alpha, \nu(\alpha)+\eps}(x'') \le \frac{\delta}{2}$, we get
    \begin{equation*}
	d_{\alpha, \nu(\alpha) + \eps}(x) \le d_{\alpha, \nu(\alpha) +
	\eps}(x') + d_{\alpha,\nu(\alpha) + \eps}(x'')\le
	\frac{\delta}{2} + \frac{\delta}{2}=\delta
    \end{equation*}
    and we are done.  Thanks to \eqref{eq:alt} we see that if $\alpha
    \in (\alpha_{n})$ describes a dense set in $[\alphamin,
    \alphamax)$ and if $\eps \in (\eps_{m})$ has limit $0$,
    together with the $b^{\alphamin - \eps}_{\infty, \infty}$ norms,
    these $\pzero$-norms completely define the topology of $\Snu$.
\end{proof}

\section{Duality}
\label{sec:duality}

We employ the usual scalar product on the sequence space $\Omega$
\begin{equation}
    \scala{x}{y} \deq \sum_{(j, k) \in \Lambda} x_{j, k}
    \overline{y_{j, k}} 
\end{equation}
to identify the topological dual $(\Snu)'$ of $\Snu$ to a sequence
space that will be revealed by Theorem~\ref{theo:dual}.  This
identification is made as follows: for all $(j,k) \in \Lambda$, let
$e^{j,k}$ be the sequence whose only non-zero component is
$e^{j,k}_{j,k} = 1$.  Given $u \in (\Snu)'$, let us define
\begin{equation}
    y \deq \sum_{j,k \in \Lambda} \overline{u(e^{j,k})} e^{j,k}.
    \nonumber
\end{equation}
This sequence $y$ indeed satisfies $u(x) = \scala{x}{y}$ because for
all $x \in \Snu$, the sum $\sum_{j,k \in \Lambda} x_{j,k} e^{j,k}$
converges to $x$ in $\Snu$ (see~\cite{ABDJ}).

\paragraph*{Remark.}

If we keep in mind that $(x_{j, k})$ represents the wavelet
coefficients of a function, the above scalar product corresponds to
the $L^2$ scalar product if and only if the said wavelets form an
orthonormal basis of $L^2([0, 1])$.
In~\cite{Aubry:zg,Aubry:2007uq,RWS}, we use $L^\infty$-normalized
orthogonal wavelets instead (which is more convenient when one uses
wavelets to study the pointwise regularity of functions); in that
setting, the $L^2$ product of functions corresponds to the
coefficients product
\begin{equation}
    \scalb{x}{y} \deq \sum_{(j, k) \in \Lambda} 2^{-j} x_{j, k}
    \overline{y_{j, k}}.
\end{equation}
The results in this section translate easily in terms of the duality
$\scalb{\cdot}{\cdot}$ by shifting the symmetry axis in Figure
\ref{fig:dualnu} by $1/2$ to the left.

\subsection{Dual profile}
\label{sec:dualpro}

\newcommand{\profilize}[1]{\llfloor #1 \rrfloor}

Let us fix a few notations.  We shall write
\begin{equation}
    \profilize{\beta} \deq
    \begin{cases}
	-\infty & \text{ if } \beta < 0 \\
	\beta & \text{ if } 0 \leq \beta \leq 1 \\
	1 & \text{ if } \beta \geq 1.  
    \end{cases}
\end{equation}
\begin{defi}
    The \emph{dual profile} of $\nu$ is the function
    \begin{equation}
	\nu' : \alpha' \mapsto \profilize{ \alpha' + \inf
	\set{\alpha}{\nu(\alpha) - \alpha > \alpha'} }.
	\label{eq:defdp}
    \end{equation}
    As in~\eqref{eq:a0} and~\eqref{eq:amax} we define the
    corresponding $\alphamin'$ and $\alphamax'$.
\end{defi}
It is easily seen that $\alphamin' = -\alphamin$ and that $ 1 -
\alphamax \leq \alphamax' \leq 1 - \alphamin$.  Graphically,
except for the discontinuities and the part where the value $1$ is
attained, the graph of $\nu'$ is obtained by horizontal symmetry, with
respect to the axis $\beta = 2\alpha$, of the graph of $\nu$
(Figure~\ref{fig:dualnu}).  Discontinuities in $\nu$ correspond to
zones with slope $1$ in $\nu'$; zones with slope $\leq 1$ in $\nu$
correspond to right-continuous discontinuities in $\nu'$ (see
Proposition~\ref{prop:dualnu} below).

\begin{figure}[ht] 
    \begin{center}
	\includegraphics[width = 0.75\textwidth]{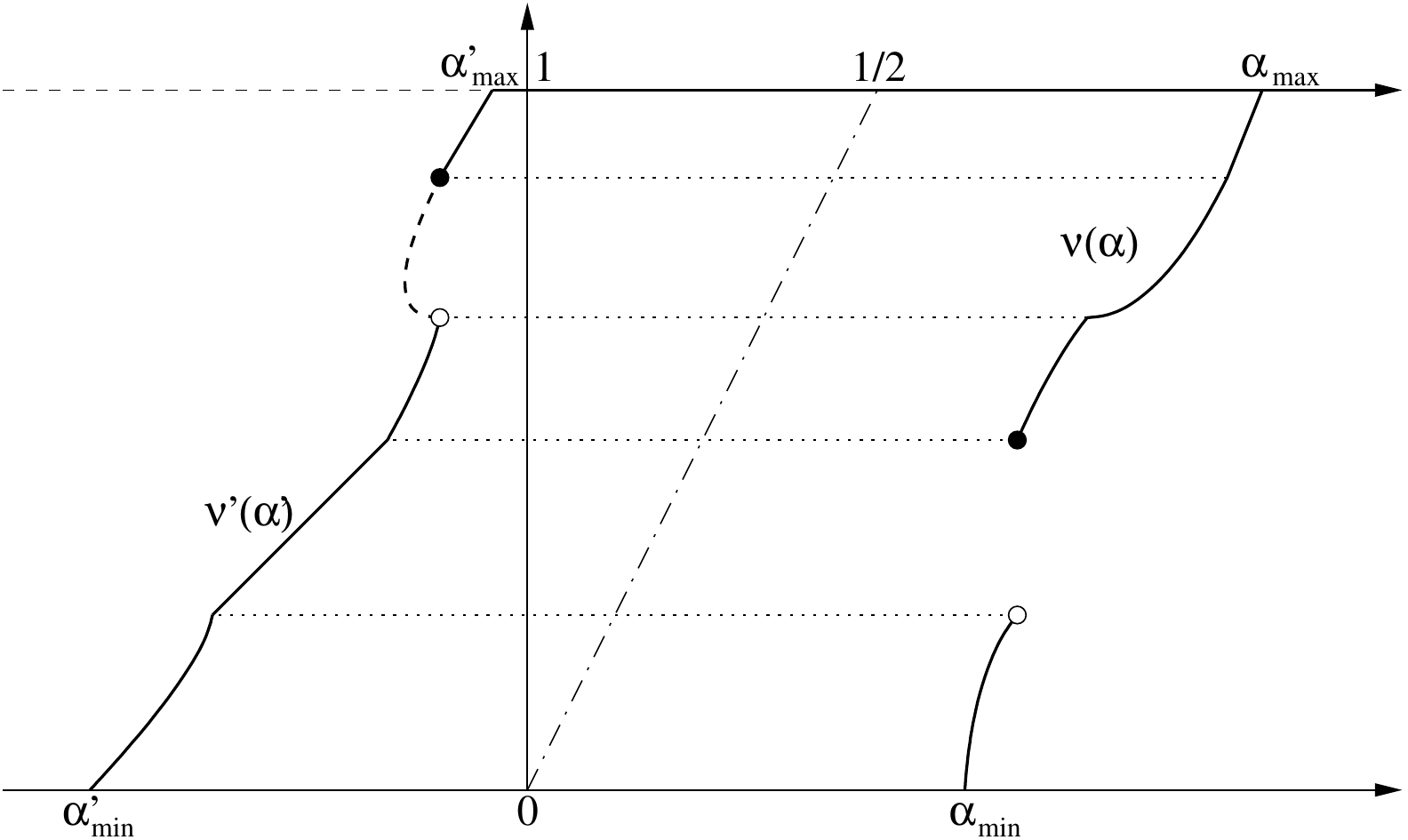}
    \end{center}
    \caption{The symmetry between $\nu$ and $\nu'$}
    \label{fig:dualnu}
\end{figure}

In the next proposition we state the less obvious properties that will
be useful to us.

\begin{propo}
    \label{prop:dualnu}
    The dual profile $\nu'$ of $\nu$ verifies
    \begin{enumerate}
	\renewcommand{\theenumi}{\roman{enumi}}
	
	\item $\nu'$ is right-continuous; \label{eq:rc}
	
	\item for all $\alpha, \alpha'$, $\alpha + \alpha' \geq
	\nu(\alpha) \wedge \nu'(\alpha')$;
	\label{eq:duality}
	
	\item if $\nu'(\alpha') = \alpha + \alpha'$ then $
	\nu'(\alpha') \leq \nu(\alpha)$;
	\label{eq:unality}
	
	\item \label{eq:slope} $\nu'$ is non-decreasing; furthermore
	for all $\alpha' \leq \alphamax'$ and $\eps \geq 0$, \\
	$\nu'(\alpha' - \eps) \leq \nu'(\alpha') - \eps$.

    \end{enumerate}
\end{propo}

\begin{proof}
    We can suppose $\alpha' \in [\alphamin', \alphamax']$, as the
    other cases are trivial.
    
    \eqref{eq:rc}: Let $\alpha'$ be fixed.  For all $\eps > 0$, there
    exists an $\alpha$ such that $\alpha' < \nu(\alpha) - \alpha$ and
    $\nu'(\alpha') + \eps \geq \alpha + \alpha'$.  Then with $\eta
    \deq \eps \wedge \paren{ \nu(\alpha) - (\alpha + \alpha') } > 0$,
    $\zeta' < \alpha' + \eta$ implies that $\zeta' < \nu(\alpha) -
    \alpha$ as well, so $\nu'(\zeta') \leq \zeta' + \alpha < \alpha' +
    \alpha + \eta \leq \nu'(\alpha') + 2 \eps$.  This proves
    right-continuity at $\alpha'$.
    
    \eqref{eq:duality}: Given $\alpha$ and $\alpha'$, if $\nu(\alpha)
    > \alpha + \alpha'$ then $\alpha \geq \inf
    \set{\tilde\alpha}{\nu(\tilde\alpha) > \tilde\alpha + \alpha'}$
    hence $\alpha + \alpha' \geq \nu'(\alpha')$; otherwise $\alpha +
    \alpha' \geq \nu(\alpha)$.
    
    \eqref{eq:unality}: If $\nu'(\alpha') = \alpha + \alpha'$ then
    $\alpha = \inf \set{\tilde\alpha}{\nu(\tilde\alpha) > \tilde\alpha
    + \alpha'}$ thus by right-continuity $\nu(\alpha) \geq \alpha +
    \alpha' = \nu'(\alpha')$.
    
    \eqref{eq:slope} is trivial.
\end{proof}

\subsection{Topological dual of $\Snu$}
\label{sec:topodual}

For $\eps > 0$ we write $\nu'_{\eps}(\alpha') \deq \nu'(\alpha' -
\eps)$.

\begin{theo}
    \label{theo:dual}
    The topological dual of $\Snu$ is
    \begin{equation}
	(\Snu)' = \bigcup_{\eps > 0} S^{\nu'_{\eps}}.  
    \end{equation}
\end{theo}

\begin{proof}
    Suppose first that $y \not \in S^{\nu'_{\eps}}$ for any $\eps > 0$
    and let us construct an $x \in \Snu$ such that $\scala{x}{y} =
    \infty$.  The hypothesis on $y$ implies that, for every
    $\varepsilon>0$, there exist $\alpha'\in\mathbb{R},\ \delta>0$
    such that $y\not\in E(\alpha',
    \nu'_{\varepsilon}(\alpha')+\delta)$; in particular, $y$ does not
    belong to any of the balls of this space.  So, given any strictly
    positive sequence $\varepsilon_n\rightarrow 0$, we thus construct
    sequences of reals $(\alpha_{n}')_{n \in \N}$ and integers
    $(j_{n})_{n \in \N}$ (the latter we can make strictly increasing)
    such that for all $n \in \N$,
    \begin{equation}
	\card\set{k}{\abs{y_{j_{n}, k}} \geq 2^{-\alpha_{n}' j_{n}}} >
	2^{\nu'(\alpha_{n}' -\eps_{n}) j_{n}}.
    \end{equation}    
    Remark that it is very possible that for some $n$,
    $\nu'(\alpha_{n}' -\eps_{n}) = -\infty$, which is equivalent to
    \begin{equation}
	\alpha_{n}'-\eps_{n} < \alphamin' = - \alphamin;
	\label{eq:nI}
    \end{equation}
    we denote by $I$ the set of such indices and by $J \deq \N
    \backslash I$ its complement.  When $n \in J$, we remark that we
    have
    \begin{equation}
	\nu'(\alpha_{n}'-\eps_{n}) = \alpha_{n}' - \eps_{n} +
	\inf\set{\alpha}{\nu(\alpha)-\alpha > \alpha_{n}' - \eps_{n}}
	\geq \alpha_{n}' - \eps_{n} + \alphamin.
	\label{eq:nJ}
    \end{equation}
    
    To construct $x$ we put
    \begin{itemize}
	
	\item for all $n \in J$, at $\ceil{2^{\nu'(\alpha_{n}'
	-\eps_{n}) j_{n}}}$ of the positions where $\abs{y_{j_{n}, k}}
	\geq 2^{-\alpha_{n}' j_{n}}$,
	\begin{equation}
	    x_{j_{n}, k} \deq 2^{-\alpha_{n} j_{n}} \frac{{y_{j_{n},
	    k}}}{\abs{y_{j_{n}, k}}} \nonumber
	\end{equation}
	with $\alpha_{n} \deq \nu'(\alpha_{n}' - \eps_{n}) -
	\alpha_{n}'$;
	
	\item for all $n \in I$, at exactly one position where
	$\abs{y_{j_{n}, k}} \geq 2^{-\alpha_{n}' j_{n}}$ we put
	$x_{j_{n},k}$ having the same expression as above, but with
	$\alpha_{n} \deq -\alpha_{n}'$;
	
    \end{itemize}
    and naturally all the other coefficients of $x$ are set equal to
    $0$.  
	
    The scalar product $\scala{x}{y}$ is divergent because for
    all $n \in \N \backslash I = J$
    \begin{equation}
	\sum_{0 \leq k < 2^{j_{n}}} x_{j_{n}, k} \overline{y_{j_{n},
	k}} \geq 2^{\nu'(\alpha_{n}' -\eps_{n}) j_{n}} 2^{(
	\alpha_{n}' - \nu'(\alpha_{n}' - \eps_{n}) ) j_{n}}
	2^{-\alpha_{n}' j_{n}} = 1  \nonumber
    \end{equation}
    whereas if $n \in I$, 
    \begin{equation}
	\sum_{0 \leq k < 2^{j_{n}}} x_{j_{n}, k} \overline{y_{j_{n},
	k}} = 2^{ \alpha_{n}' j_{n}} 2^{-\alpha_{n}' j_{n}} = 1.
	\nonumber
    \end{equation}
    
    It remains to prove that $x$ belongs to $\Snu =
    \bigcap\limits_{\alpha \in \R, \eps > 0} E(\alpha, \nu(\alpha) +
    \eps)$.  For this we consider two cases.

    \underline{Firstly, if $\alpha < \alphamin$:} In this situation,
    we have to show that $\sup_{j,k} 2^{\alpha j} \abs{x_{j,k}} <
    \infty$ or, equivalently, that
    \begin{equation*}
	(i) \quad \sup_{n\in I} 2^{\alpha j_{n}} 2^{\alpha_{n}' j_{n}}
	< \infty \quad \text{ and } \quad (i\!i) \quad \sup_{n \in J}
	2^{\alpha j_{n}} 2^{(\alpha_{n}' - \nu'(\alpha_{n}'-\eps_{n})
	) j_{n}} < \infty.
    \end{equation*}    
    When $n \in I$, by~\eqref{eq:nI} we have $\alpha_{n}' + \alpha <
    \eps_{n} + \alpha - \alphamin$, whereas when $n \in J$,
    by~\eqref{eq:nJ} we get $\alpha + \alpha_{n}' - \nu'(\alpha_{n}' -
    \eps_{n}) \leq \eps_{n} + \alpha - \alphamin$.  Since $\eps_{n}
    \to 0$ this quantity becomes negative when $n$ is large enough, so
    both (i) and (ii) hold.

    \underline{Secondly, if $\alpha \geq \alphamin$:} In that case,
    let $\beta \deq \nu(\alpha) + \eps$.  We have to show that there
    exists $C > 0$ such that for all $n$,
    \begin{equation*}
	\card\set{k}{\abs{x_{j_{n},k}} \geq C 
	    2^{-\alpha j_{n}}} \leq C 2^{\beta j_{n}}.
    \end{equation*}    
    When $n \in I$ this is trivial, since $\beta > 0$ and there is
    only one non-zero coefficient in $x$ at scale $j_{n}$.  When $n
    \in J$, either $\alpha_{n} = \nu'(\alpha_{n}'-\eps_{n}) -
    \alpha_{n}' > \alpha$, and the above cardinal is zero, or
    $\nu'(\alpha_{n}'-\eps_{n}) - \alpha_{n}' \leq \alpha$.  In that
    last case, using the right-continuity of $\nu$ we get that
    $\nu(\alpha_{n}+\eps_{n}) \leq \beta$ for $n$ large enough;
    using~\eqref{eq:unality} of Proposition~\ref{prop:dualnu} we have
    $\nu'(\alpha_{n}'-\eps_{n}) \leq \nu(\alpha_{n}+\eps_{n}) \leq
    \beta$ and the conclusion follows.  We have proved that $y \not
    \in \bigcup_{\eps > 0} S^{\nu'_{\eps}}$ cannot belong to the dual
    of $\Snu$.
    
    \newcommand{\Lprime}{L}
    \newcommand{\alphap}{\alpha'}
    
    Conversely, let $\eps > 0$ and $y \in S^{\nu'_{2\eps}}$.  We
    construct $L \deq \ceil{\frac{4}{\eps}}$ and for $-1 \leq l \leq
    L$, $\alpha_{l} \deq \alphamin + \frac{\eps}{4}l$ and $\nu_{l}
    \deq \nu(\alpha_{l}) + \frac{\eps}{4}$.  Similarly, for $-1 \leq
    l' \leq \Lprime$, $\alphap_{l'} \deq \alphamin' + 2 \eps +
    \frac{\eps}{4}l'$ and $\mu_{l'} \deq \nu'(\alphap_{l'} - 2\eps) +
    \frac{\eps}{4}$.
    
    Let $U \deq \bigcap_{l = -1}^{L} U_{l}$, where $U_{l}$ is the open
    unit ball in $E(\alpha_{l}, \nu_{l})$ and fix an $A >
    \max\limits_{{-1 \leq l' \leq \Lprime}}{d_{\alphap_{l'},
    \mu_{l'}}(y)}$.    
    We split any $x \in U$ as $x = \sum_{l = 0}^{L+1} x_{l}$, where
    for $0 \leq l \leq L$, $x_{l}$ receives the coefficients $
    2^{-\alpha_{l}j} < \abs{x_{j, k}} \leq 2^{-\alpha_{l-1}j}$ and
    $x_{L+1}$ receives the coefficients $\abs{x_{j, k}} \leq
    2^{-\alpha_{L}j}$ (since $\nu(\alpha_{-1}) = -\infty$ and $x \in
    U_{-1}$, there is no coefficient $\abs{x_{j, k}} >
    2^{-\alpha_{-1}j}$).
    
    We do the same to $y$, writing $y = \sum_{l' = 0}^{\Lprime+1}
    y_{l'}$, where for $0 \leq l' \leq \Lprime$, $y_{l'}$ receives the
    coefficients $A 2^{-\alphap_{l'}j} < \abs{y_{j, k}} \leq A
    2^{-\alphap_{l'-1} j}$ and $y_{\Lprime+1}$ receives the
    coefficients $\abs{y_{j, k}} \leq A 2^{-\alphap_{\Lprime}j}$
    (same remark as above about the coefficients $\abs{y_{j, k}} > A
    2^{-\alphap_{-1} j}$: there are none because $\mu_{-1} =
    \nu'(\alphamin'-\frac{\eps}{4}) = -\infty$).
    
    The proof now boils down to studying each term of
    \begin{equation}
	\scala{x}{y} = \sum_{l = 0}^{L+1} \sum_{l' = 0}^{\Lprime+1}
	\scala{x_{l}}{y_{l'}}.
    \end{equation}
    Four cases can be distinguished.
    
    \underline{Firstly, if $0 \leq l \leq L$, $0 \leq l' \leq \Lprime$
    and $\nu(\alpha_{l}) \leq \nu'(\alphap_{l'} - \eps)$:} Then by
    Proposition~\ref{prop:dualnu}~\eqref{eq:duality} applied to
    $\alpha_{l}$ and $\alphap_{l'}-\eps$, this means that
    \begin{math}
	\alpha_{l} + \alphap_{l'} \geq \nu(\alpha_{l}) + \eps =
	\nu_{l} + \frac{3\eps}{4}.  \nonumber
    \end{math}
    At scale $j$ in $x_{l}$, there are less than $2^{\nu_{l} j}$
    non-zero coefficients which are bounded by $2^{-\alpha_{l-1}j}$,
    so
    \begin{equation*}
	\abs{\scala{x_{l}}{y_{l'}}} \leq A \sum_{j\in\Nz} 2^{(\nu_{l}
	-\alpha_{l-1}-\alphap_{l'-1}) j} \leq A \sum_{j\in\Nz}
	2^{-\frac{\eps}{4}j}.
    \end{equation*}
    
    \underline{Secondly, if $0 \leq l \leq L$, $0 \leq l' \leq
    \Lprime$ and $\nu(\alpha_{l}) \geq \nu'(\alphap_{l'} - \eps)$:}
    Then by Proposition~\ref{prop:dualnu}~\eqref{eq:duality} once
    again we get
    \begin{math}
	\alpha_{l} + \alphap_{l'} \geq \nu(\alpha_{l}) \wedge
	\nu'(\alphap_{l'}) \geq \nu'(\alphap_{l'} - \eps) \geq
	\nu'(\alphap_{l'} - 2 \eps) + \eps = \mu_{l'} +
	\frac{3\eps}{4}.  \nonumber
    \end{math}
    At scale $j$, in $y_{l'}$ there are less than $A 2^{\mu_{l'} j}$
    non-zero coefficients which are bounded by $A 2^{-\alphap_{l'-1}
    j}$, so
    \begin{equation*}
	\abs{\scala{x_{l}}{y_{l'}}} \leq A^2 \sum_{j\in\Nz}
	2^{(\mu_{l'} -\alpha_{l-1}-\alphap_{l'-1}) j} \leq A^2
	\sum_{j\in\Nz} 2^{-\frac{\eps}{4}j}.
    \end{equation*}
    
    \underline{Thirdly, if $l = L+1$ and $0 \leq l' \leq \Lprime+1$:}
    Since $\alphap_{l'-1} \geq -\alphamin + \frac{7 \eps}{4}$ and
    $\alpha_{L} \geq \alphamin + 1$, we get by a direct computation
    \begin{equation*}
	\abs{\scal{x_{L+1}}{y_{l'}}} \leq A \sum_{j \in \Nz}
	2^{(1-\alpha_{L} - \alphap_{l'-1}) j} \leq A \sum_{j \in \Nz}
	2^{-\frac{7\eps}{4}j}.
    \end{equation*}    
    
    \underline{Finally, if $0 \leq l \leq L$ and $l' = \Lprime+1$:}
    Since $\alpha_{l-1} \geq \alphamin - \frac{\eps}{4}$ and
    $\alphap_{\Lprime} \geq \alphamin' + 1 + 2 \eps$, we obtain
    \begin{equation*}
	\abs{\scal{x_{l}}{y_{\Lprime+1}}} \leq A^{2} \sum_{j \in \Nz}
	2^{(1-\alpha_{l-1} - \alphap_{\Lprime}) j} \leq A^{2} \sum_{j
	\in \Nz} 2^{-\frac{7\eps}{4}j}.
    \end{equation*}
    
    In the end, $\scal{x}{y}$ is bounded on $U$, the bound depending
    only on $A$ and $\eps$ (that is, only on $y$).  This proves that
    the linear form $x \mapsto \scala{x}{y}$ is continuous.
    
\end{proof}

\subsection{Strong topology on $(\Snu)'$}
\label{sec:topodl}

In the previous theorem, the dual of $\Snu$ has been algebraically
identified to a union of spaces $S^{\nu'_{\eps}}$, for $\eps > 0$, or
equivalently to a countable union of spaces $S^{\nu_{m}'}$ for
$\nu_{m}' \deq \nu_{\eps_{m}}'$, $\eps_{m} \searrow 0$.  As such, it can
be endowed with the inductive limit topology on this union, now
written $\ind_{m} S^{\nu_{m}'}$.  We shall now see that, at least 
when the convexity index~\eqref{eq:pzero}
$\pzero = 1$, this topology is actually the same as the strong
topology on the dual (then written $(\Snu)'_{b}$ in the standard
notation), that is, the topology of uniform convergence on the bounded
sets of $\Snu$.

Before that, recalling that a Montel space is a barrelled tvs in which
every bounded set is relatively compact, we give another remarkable
property of $\Snu$.

\begin{propo}
    If $\pzero =1$, then $\Snu$ is a Fréchet-Montel space.
\end{propo}

\begin{proof}
    In~\cite[Proposition 6.2]{ABDJ} we obtained the following
    characterization: A subset $K$ of $\Snu$ is compact if and only if
    it is closed and bounded for each of the distances $d_{\alpha_{n},
    \nu(\alpha_{n}) + \eps_{m}}$ (cf.~\eqref{eq:dab},~\eqref{eq:alt}).
    The special structure of the topology of $\Snu$ show that the
    latter condition is in fact equivalent to saying that $K$ is
    bounded in $\Snu$.  So, when $\pzero = 1$, the tvs $\Snu$ is a
    Fréchet space in which the bounded sets are relatively compact.
\end{proof}

\begin{coro}
    $\Snu$ is reflexive if and only if $\pzero = 1$.
\end{coro}

\begin{proof}
    Every Fréchet-Montel space is reflexive; conversely a reflexive
    space is by definition necessarily locally convex.
\end{proof}

\begin{theo}
    \label{theo:topodl}
    If $\pzero =1$, then topologically $(\Snu)'_{b} = \ind_{m}
    S^{\nu_{m}'}$.
\end{theo}

\begin{proof}
    The fact that the canonical injection $S^{\nu'_m}\ \rightarrow
    \left(\Snu\right)'_b$ is continuous for every $m$ is obtained
    using the characterization of the bounded sets of $\Snu$ and a
    part of the proof leading to the algebraic description of the
    dual.  Indeed, if $m$ is fixed and if $B$ is a bounded set of
    $\Snu$, we replace the unit balls $U_l$ by balls of radius $R>0$
    so that $B\subset U \deq \bigcap_{l=-1}^LU_l$ (with $L$ depending
    on $m$).  The same proof then shows that $\scala{x}{y} \rightarrow
    0$ when $y\rightarrow 0$ in $S^{\nu'_m}$, uniformly on $B$.  This
    proves that the inductive limit topology is stronger than the
    strong topology (and as a consequence, the inductive limit is
    Hausdorff).

    To prove that the topologies are in fact equivalent, we use the
    closed graph theorem of De Wilde (see~\cite{Jarchow:1981sp}
    or~\cite{Meise:1997fk}) for the identity map from the strong dual
    into the inductive limit: the strong dual is ultrabornological
    (since it is the strong dual of a Fréchet-Montel space) and the
    inductive limit is a webbed space.  Since the identity has a
    closed graph, it is continuous.
\end{proof}

So far the case $\pzero < 1$ remains an open problem (the missing point
is to show that the strong dual is ultrabornological or Baire, in
order to be able to apply the closed graph theorem).

\subsection{Dual of an intersection of Besov spaces}
\label{sec:dualB}

We conclude with an application to a particular case.  As we recalled
earlier, when $\nu$ is concave we have
\begin{equation}
    S^{\nu} = \bigcap_{\eps > 0, p > 0} b^{\eta(p)/p - \eps}_{p,
    \infty} \nonumber
\end{equation}
with~\eqref{eq:concon} $\eta(p) \deq \inf_{\alpha \geq \alphamin}
\paren{ \alpha p - \nu(\alpha) + 1}$.  If we invert this
Fenchel-Legendre transform we obtain
\begin{equation}
    \nu(\alpha) = \inf_{p > 0} \paren{ \alpha p - \eta(p) + 1 }.
    \label{eq:nocnoc}
\end{equation}
In that case the dual profile $\nu'$ is convex on $[\alphamin',
\alphamax')$ and can be directly computed from $\eta$ as shown in
Proposition~\ref{prop:db} below.  Then by Theorems~\ref{theo:dual}
and~\ref{theo:topodl} we know all about the strong topological dual of
this intersection of Besov spaces.

We shall keep in mind that the concavity of $\nu$ implies that it is now
continuous and that the right derivative $\partial^{+} \nu(\alpha)$
exists for all $\alpha \geq \alphamin$.
\begin{lemm}
    \label{lemm:plouf}
    If $\nu(\alpha) = \alpha p - \eta(p) + 1$ for some $p > 0$, then
    $\partial^{+}\nu(\alpha) \leq p$.
\end{lemm}

\begin{proof}
    Let $h > 0$, and observe that 
    \begin{equation*}
	\nu(\alpha + h) = \inf_{\tilde p > 0} \paren{(\alpha + h)
	\tilde p - \eta(\tilde p) + 1} \leq (\alpha + h) p - \eta(p) +
	1 = \nu(\alpha) + h p
    \end{equation*}
    and the conclusion follows readily.
\end{proof}

\begin{propo}
    \label{prop:db}
    If $\nu$ is concave and $\eta$ is its conjugate, then the function
    $\eta'$ defined for $p' > 1$ by
    \begin{equation}
	\eta'(p') \deq (p'-1) \paren{1-\eta\paren{\frac{p'}{p'-1}}} +
	1
	\label{eq:edl}
    \end{equation}
    is convex and for $\alpha' \in [\alphamin', \alphamax')$ we have
    \begin{equation}
	\nu'(\alpha') = \sup_{p' > 1} \paren{ \alpha' p' - \eta'(p') +
	1 }.
	\label{eq:ndl}
    \end{equation}
\end{propo}

The appearance of~\eqref{eq:edl} should not be surprising if one
notices, as an easy consequence of~\eqref{eq:pnb} and Hölder's
inequality, that when $p > 1$ the dual of $b_{p, 1}^{\eta(p)/p}$ is
just $b_{p', \infty}^{\eta'(p')/p'}$, with $\frac{1}{p} + \frac{1}{p'}
= 1$.

\begin{proof}
    Let $p_{1}, p_{2} > 1$ and let $p_{1}' \deq
    \frac{p_{1}}{p_{1}-1}$, $p_{2}' \deq \frac{p_{2}}{p_{2}-1}$.  Let
    $p' \deq \frac{p_{1}' + p_{2}'}{2}$, and $p'' \deq \frac{p'}{p'-1}
    = \frac{p_{2}-1}{p_{1}+p_{2}-2} p_{1} + \frac{p_{1}-1}{p_{1} +
    p_{2} - 2} p_{2}$.  We have
    \begin{align*}
	\eta'(p') & = (p'-1)(1-\eta(p'')) + 1 \\
	\intertext{using the concavity of $\eta$,} &\leq (p'-1)
	\paren{1- \frac{p_{2}-1}{p_{1}+p_{2}-2} \eta(p_{1}) -
	\frac{p_{1}-1}{p_{1} + p_{2} - 2} \eta(p_{2})} + 1 \\
	&\leq \frac{p'-1}{p_{1}+p_{2}-2}
	\paren{(p_{2}-1)(1-\eta(p_{1})) + (p_{1}-1)(1-\eta(p_{2}))} +
	1 \\
	&\leq \frac{p_{1}'-1}{2} (1-\eta(p_{1})) +
	\frac{p_{2}'-1}{2}(1-\eta(p_{2})) + 1 \\
	&\leq \frac{\eta'(p_{1}') + \eta'(p_{2}')}{2}
    \end{align*}
    so $\eta'$ is indeed convex.

    Let us now prove~\eqref{eq:ndl}.  We start from the
    definition~\eqref{eq:defdp} of $\nu'(\alpha')$, having fixed an
    $\alpha' \in [\alphamin', \alphamax')$.  Then $
    \set{\alpha}{\nu(\alpha)-\alpha > \alpha'} \neq \emptyset$
    (because its infimum is finite).  Let us write
    \begin{math}
	\tilde\alpha \deq \inf\set{\alpha}{\nu(\alpha)-\alpha >
	\alpha'} \nonumber
    \end{math}
    and consider two cases.
    
    \underline{If $\nu(\tilde\alpha) - \tilde\alpha > \alpha'$:}
    Necessarily $\tilde\alpha = \alphamin$ (otherwise a contradiction
    is easily reached using the continuity of $\nu$) and
    $\nu'(\alpha') = \alpha' - \alphamin' = \alpha' + \tilde\alpha$.
    The concavity of $\nu$ implies that as soon as $p \geq
    \partial^{+} \nu(\tilde\alpha)$, $\alpha \mapsto \alpha p -
    \nu(\alpha) + 1$ is increasing on $[\tilde\alpha,\infty)$ and
    \eqref{eq:concon} becomes $\eta(p) = \tilde\alpha p -
    \nu(\tilde\alpha) + 1$.  This means that if $p' = \frac{p}{p-1}$
    is close enough to $1$, $\eta'(p') = (p'-1) \nu(\tilde\alpha) -
    \tilde\alpha p' + 1$ and $\alpha' p' - \eta'(p') + 1 = p'(\alpha'
    +\tilde\alpha - \nu(\tilde\alpha)) + \nu(\tilde\alpha)$.  The
    function $p' \mapsto \alpha' p' - \eta'(p') + 1$, being concave,
    is thus strictly decreasing on $(1,\infty)$ and the supremum on
    the right-hand side of \eqref{eq:ndl} can be computed as
    \begin{align*}
	\sup_{p' > 1} \paren{ \alpha' p' - \eta'(p') + 1 } &= \lim_{p'
	\to 1} \alpha' p' - \eta'(p') + 1 \\
	&= \alpha' + \tilde\alpha.
    \end{align*}
    
    \underline{If $\nu(\tilde\alpha) - \tilde\alpha = \alpha'$:} Pick
    any $\alpha''$ such that $\nu(\alpha'') - \alpha'' > \alpha'$ and
    observe that $\nu(\alpha'') - \nu(\tilde\alpha) > \alpha'' -
    \tilde\alpha$; by concavity again this implies that
    $\partial^{+}\nu(\alpha) > 1$ when $\alpha > \tilde\alpha$ is
    close enough.  Thus by Lemma~\ref{lemm:plouf}, for these $\alpha$,
    the infimum in~\eqref{eq:nocnoc} is reached for a $p > 1$.
    Another consequence is that $\nu'(\alpha') = \alpha' + \inf
    \set{\alpha}{\nu(\alpha) - \alpha \geq \alpha'}$.  It follows that
    \begin{align*}
	\nu'(\alpha') &= \alpha' + \inf \set{\alpha}{\inf_{p>1}
	(\alpha p - \eta(p) + 1) - \alpha \geq \alpha'} \\
	& = \alpha' + \inf\set{\alpha}{\forall p > 1,
	\alpha (p - 1) - \eta(p) + 1 \geq \alpha'} \\
	& = \alpha' + \inf\set{\alpha}{\forall p > 1, \alpha \geq
	\frac{\alpha' + \eta(p) - 1}{p-1}} \\
	& = \alpha' + \sup_{p > 1} \paren{ \frac{\alpha' + \eta(p) -
	1}{p-1} }.
    \end{align*}
    On the other hand, changing the variable $p'$ into $p \deq
    \frac{p'}{p'-1}$ yields
    \begin{align*}
	\sup_{p' > 1} \paren{ \alpha' p' - \eta'(p') + 1} & = \sup_{p
	> 1} \paren{\frac{\alpha' p}{p - 1} - \frac{1-\eta(p)}{p-1} }
	\\
	& = \alpha' + \sup_{p > 1} \paren{ \frac{\alpha' + \eta(p) -
	1}{p-1} }
    \end{align*}
    and the proposition is proved.
    
\end{proof}

\end{document}